\documentclass[12pt] {amsart}
\usepackage{amsmath,amssymb,amsthm,amsfonts,amscd,amsopn}
\usepackage{eucal, mathrsfs}
\usepackage{enumitem,comment}
\usepackage{hyperref}
\usepackage[all]{xy}\xyoption{dvips}
\theoremstyle{plain}
\newtheorem{thm}{Theorem}[section]

\newtheorem{lemma}[thm]{Lemma}
\newtheorem{lem}[thm]{Lemma}

\newtheorem{cor}[thm]{Corollary}

\newtheorem{prop}[thm]{Proposition}

\theoremstyle{definition}
\newtheorem{rem}[thm]{Remark}

\newtheorem{rmk}[thm]{Remark}

\newtheorem{defn}[thm]{Definition}
\newtheorem{ex}[thm]{Example}
\newtheorem{claim}[thm]{Claim}

\newtheorem{notation}[thm]{Notation}
\theoremstyle{remark}
\newtheorem*{claim*}{Claim}
\newtheorem{subclaim}[equation]{Claim}
\numberwithin{equation}{thm}

\renewcommand{\mathcal}{\mathscr}

\newcommand{\A}{{\mathbb A}}

\newcommand{\F}{{\mathcal F}}

\newcommand{\N}{{\mathbb N}}

\renewcommand{\P}{{\mathbb P}}

\newcommand{\sI}{\mathscr{I}}

\newcommand{\sL}{\mathscr{L}}

\newcommand{\sO}{\mathscr{O}}

\DeclareMathOperator{\Proj}{Proj}
\DeclareMathOperator{\Spec}{Spec}

\DeclareMathOperator{\charact}{char}
\DeclareMathOperator{\id}{id}
\title[On the failure of generic vanishing]{Generic vanishing fails for singular
  varieties and in characteristic $p>0$}
\author{Christopher D. Hacon}
\address{Department of Mathematics, University of Utah, 155 South 1400 East,
Salt Lake City, UT 48112-0090, USA}
\email{hacon@math.utah.edu}
\author{S\'andor J Kov\'acs}
\address{University of Washington, Department of Mathematics, Box 354350,
Seattle, WA 98195-4350, USA}
\email{skovacs@uw.edu}
\usepackage[usenames]{color}
\definecolor{tomato}{RGB}{180,62,39}
        
\definecolor{forrest}{RGB}{81,133,49}    
        
\definecolor{lighttomato}{RGB}{253,65,65}
        
\definecolor{lightforrest}{RGB}{145,237,87}    
        
\definecolor{mygreen}{RGB}{40,104,69}    
        
\definecolor{mygreen2}{RGB}{3,149,39}    
        
\definecolor{darkolivegreen}{RGB}{102,118,75}
        
\definecolor{cranegreen}{RGB}{102,118,75}
        
\definecolor{myverydarkblue}{RGB}{0,0,20}
        
\definecolor{mydarkblue}{RGB}{10,92,153}
        
\definecolor{myblue}{RGB}{57,222,186}
        
\definecolor{pinkish}{RGB}{213,83,222}
        
\definecolor{colD}{RGB}{213,83,222}
        
\definecolor{defb}{RGB}{213,83,222}
        
\definecolor{goldenrod}{RGB}{225,115,69}
        
\definecolor{mauve}{RGB}{224, 176, 255}
        
\definecolor{fuchsia}{RGB}{255, 0, 255}
        
\definecolor{lavender}{RGB}{230, 230, 250}
        
\definecolor{grey}{RGB}{180,180,180}
        
\definecolor{lightgrey}{RGB}{220,220,220}

\makeatletter
\setitemize[1]{leftmargin=*,parsep=0em,itemsep=0.125em,topsep=0.125em}
\makeatother
\newcommand\hkL{}

\dedicatory{Dedicated to Rob Lazarsfeld on the occasion of his sixtieth birthday.}
\begin{document}
\maketitle
 
\section{Introduction}

In recent years there has been considerable interest in understanding the geometry of
irregular varieties, i.e., varieties admitting a nontrivial morphism to an abelian
veriety.  One of the central results in the area is the following result conjectured
by M. Green and R. Lazarsfeld (cf. \cite[6.2]{GL91}) and proven in \cite{Hacon04} and
\cite{PP09}.

\begin{thm}\label{GV}
  Let $\lambda :X\to A$ be a generically finite (onto its image) morphism from a compact
  K\"ahler manifold to a complex torus. If $\mathcal L \to X\times {\rm Pic }^0(A)$
  is the universal family of topologically trivial line bundles, then
  $$
  R^i{\pi _{ {\rm Pic }^0(A) * }}\mathcal L=0\qquad {\rm for}\ i<n.
  $$  
\end{thm}

At first sight, the above result appears to be quite technical however it has many
concrete applications (see for example \cite{CH11}, \cite{JLT11} and \cite{PP09}).
In this paper we will show that \eqref{GV} does not generalize to characteristic
$p>0$ or to singular varieties in characteristic $0$.

\begin{notation}\label{not:ell}
  Let $A$ be an abelian variety over an algebraically closed field $k$, $\widehat A$
  its dual abelian variety, $\mathcal P$ the normalized Poincar\'e bundle on $A\times
  \widehat A$ and $p _{\widehat A}:A\times \widehat A\to \widehat A$ the projection.
  Let $\lambda :X\to A$ be a projective morphism, $\pi _{\widehat A}:X\times \widehat
  A\to \widehat A$ the projection and $\mathcal L:=(\lambda\times{\id}_{\widehat
    A})^*\mathcal P$ where $(\lambda\times{\id}_{\widehat A}):X\times \widehat A\to
  A\times \widehat A$ is the product morphism.
\end{notation}

\begin{thm}\label{t-1}
  Let $k$ be an algebraically closed field. Then, using
  the notation in \eqref{not:ell}, there exist a projective variety $X$ over $k$ such
  that
  \begin{itemize}
  \item if $\charact k=p>0$, then $X$ is smooth, and 
  \item if $\charact k=0$, then $X$ has isolated Gorenstein log canonical
    singularities,
  \end{itemize}
  and a separated projective morphism to an abelian variety $\lambda :X\to A$ which is generically finite onto its
  image and such that
  $$R^i{\pi _{\widehat A}}_*\mathcal L\ne 0\qquad {\rm for\ some}\ 0\leq i<n.$$  
\end{thm}

\begin{rmk}
  Due to the birational nature of the statement, \eqref{GV} trivially generalizes to the case of
  $X$ having only rational singularities. Arguably Gorenstein log canonical
  singularities are the simplest examples of singularities that are not rational.  Therefore the
  characteristic $0$ part of \eqref{t-1} may be interpreted as saying that generic
  vanishing does not extended to singular varieties in a non-trivial way.
\end{rmk}

\begin{rmk}
  Note that \eqref{t-1} seems to contradict the main result of \cite{Pareschi03}.
\end{rmk}
\hskip.3cm

\noindent{\bf Acknowledgments.}
The authors would like to thank G. Pareschi, A. Langer and M. Popa for useful
discussions and comments.  The first named author was partially supported by NSF
research grant DMS-0757897 and DMS-1300750 and a grant from the Simons foundation.
The second named author was supported in part by NSF Grants DMS-0856185 and
DMS-1301888 and the Craig McKibben and Sarah Merner Endowed Professorship in
Mathematics at the University of Washington.

\section{Preliminaries}\label{s-pre}

Let $A$ be a $g$-dimensional abelian variety over an algebraically closed field $k$,
$\widehat A$ its dual abelian variety $p_A$ and $p_{\widehat A}$ the projections of
$A\times \widehat A$ onto $A$ and $\widehat A$, and $\mathcal P$ the normalized
Poincar\'e bundle on $A\times \widehat A$. We denote by $\mathbf R \widehat S:
\mathbf D(A)\to \mathbf D(\widehat A)$ the usual Fourier-Mukai functor given by
$\mathbf R \widehat S(\F )=\mathbf R {p_{\widehat A}}_*(\hkL p_A^*\F \otimes \mathcal
P)$ cf.\ \cite{Mukai81}.  There is a corresponding functor $\mathbf R S: \mathbf
D(\widehat A)\to \mathbf D( A)$ such that $$\mathbf R S\circ \mathbf R \widehat
S=(-1_A)^*[-g]\qquad {\rm and }\qquad \mathbf R \widehat S\circ \mathbf R
S=(-1_{\widehat A})^*[-g].$$

\begin{defn}
  An object $F\in \mathbf D(A)$ is called \emph{WIT-$i$} if $R ^j \widehat S (F)=0$
  for all $j\ne i$.  In this case we use the notation $\widehat F= R ^i \widehat S
  (F)$.
\end{defn}
Notice that if $F$ is a WIT-$i$ coherent sheaf (in degree $0$), then $\widehat F$ is
a WIT-$(g-i)$ coherent sheaf (in degree $i$) and $F\simeq (-1_A)^*R^{g-i}S (\widehat
F)$.

One easily sees that if $F$ and $G$ are arbitrary objects, then
$${\rm Hom}_{\mathbf D (A)}(F,G)={\rm Hom}_{\mathbf D (\widehat   A)}(\mathbf R
\widehat S F,\mathbf R \widehat S G).$$ An easy consequence (cf.\
\cite[2.5]{Mukai81}) is that if $F$ is a WIT-$i$ sheaf and $G$ is a WIT-$j$ sheaf (or if $F$ is a WIT-$i$ locally free sheaf and $G$ is a WIT-$j$ object -- not
necessarily a sheaf), then

\begin{multline}
  \label{eq-0}
  {\rm Ext} ^k_{\mathcal O _A}(F,G) \simeq {\rm Hom}_{\mathbf D(A)}(F,G[k])\simeq \\
  \simeq {\rm Hom}_{\mathbf D(\widehat A)}(\mathbf R\widehat S F,\mathbf R\widehat S
  G[k])=\\
  = {\rm Hom}_{\mathbf D(\widehat A)}(\widehat F [-i],\widehat G[k-j])\simeq {\rm
    Ext} ^{k+i-j}_{\mathcal O _{\widehat A}}(\widehat F,\widehat G).
\end{multline}

Let $L$ be any ample line bundle on $\widehat A$, then $ \mathbf RS(L)= R^0 S(L)=
\widehat L$ is a vector bundle on $A$ of rank $h^0(L)$.  For any $x\in A$, let
$t_x:A\to A$ be the translation by $x$ and let $\phi _L:\widehat A \to A$ be the
isogeny determined by $\phi _L (\widehat x)=t^*_{\widehat x}L\otimes L^\vee $, then
$\phi _L^*(\widehat{L})=\bigoplus _{h^0(L)}L^\vee$.

Let $\lambda :X\to A$ be a projective morphism of normal varieties, and $\mathcal
L=(\lambda\times {\id}_{\widehat A})^*\mathcal P$. We let $\mathbf R\Phi :\mathbf D
(X)\to \mathbf D(\widehat A)$ be the functor defined by $\mathbf R\Phi (F)=\mathbf
R{\pi _{\widehat A}}_*(\hkL\pi _X^*F\otimes \mathcal L)$ where $\pi _X$ and $\pi
_{\widehat A}$ denote the projections of $X\times \widehat A$ on to the first and
second factor.  Note that
\begin{multline}\label{eq-1}
  \mathbf R\Phi (F)=\mathbf R{\pi _{\widehat A}}_*(\hkL\pi_X^*F\otimes \mathcal
  L)\simeq^1 \\
  \simeq \mathbf R{p _{\widehat A}}_*\mathbf R(\lambda\times {\id}_{\widehat
    A})_*(\hkL\pi_X^*F\otimes (\lambda\times {\id}_{\widehat A})^*\mathcal P)\simeq^2
  \\
  \simeq \mathbf R{p _{\widehat A}}_*\left(\mathbf R(\lambda\times {\id}_{\widehat
      A})_*(\hkL\pi_X^*F)\otimes\mathcal P\right)\simeq^3
  \\
  \simeq \mathbf R{p _{\widehat A}}_*(\hkL p_A^*\mathbf R\lambda_*F\otimes \mathcal
  P)\simeq \mathbf R \widehat S (\mathbf R\lambda_*F ),
\end{multline} 
where $\simeq^1$ follows by composition of derived functors \cite[II.5.1]{MR0222093},
$\simeq^2$ follows by the projection formula \cite[II.5.6]{MR0222093}, and $\simeq^3$
follows by flat base change \cite[II.5.12]{MR0222093}.

We also define $\mathbf R\Psi :\mathbf D (\widehat A)\to \mathbf D(X)$ by $\mathbf
R\Psi (F)=\mathbf R{\pi _X}_*(\hkL\pi _{\widehat A}^*F\otimes \mathcal L)$.
Notice that if $F$ is a locally free sheaf, then $\hkL\pi _{\widehat
  A}^*F\otimes \mathcal L$ is also a locally free sheaf. In particular, for any $i\in \mathbb Z$, we have that
\begin{equation}
  \label{eq:Rg}
  R^i\Psi(F)\simeq R^i{\pi_X}_*(\pi_{\widehat A}^*F\otimes\sL).
\end{equation}

We will need the following fact (which is also proven during the proof of Theorem B of \cite{PP11}).
\begin{lemma}\label{l-c} 
  Let $L$ be an ample line bundle on $\widehat A$, then
  $$\mathbf R \Psi (L^{\vee})=R^g\Psi (L^{\vee})=\lambda^*\widehat {L^{\vee}}.$$
\end{lemma}
\begin{proof}
  Since $L$ is ample, $H^i(\widehat A, L^{\vee }\otimes \mathcal L _x)=H^i(\widehat A, L^{\vee }\otimes \mathcal P _{\lambda(x)})=0$ for $i\ne g$
where $\mathcal P _{\lambda(x)}=\mathcal P |_{\lambda(x)\times \widehat A}$ and
  $\mathcal L _{x}=\mathcal L |_{x\times \widehat A}$ are isomorphic. By cohomology and
  base change $\mathbf R \Psi (L^{\vee })=R^g\Psi (L^{\vee })$ (resp.  $\widehat {L^{\vee }}$)
  is a vector bundle of rank $h^g(\widehat A, L^{\vee })$ on $X$ (resp. on $A$).

  The a natural transformation 
  $\id_{A\times \widehat A}\to (\lambda\times \id_{\widehat A})_*(\lambda\times
  \id_{\widehat A})^*$ induces a natural morphism,
  $$
  \widehat {L^{\vee }}= R^g{p_A}_*(p_{\widehat A}^*L^{\vee }\otimes \mathcal P)\to
  R^g{p_A}_*(\lambda\times {\id}_{\widehat A})_*(\pi _{\widehat A}^*L^{\vee }\otimes \mathcal
  L).
  $$ 

  Let $\sigma=p_A\circ (\lambda\times {\id}_{\widehat A})=\lambda\circ\pi_X$.  By the
  Grothendieck spectral sequence associated to ${p_A}_*\circ (\lambda\times
  {\id}_{\widehat A})_*$ there exists a natural morphism
  $$
  R^g{p_A}_*(\lambda\times {\id}_{\widehat A})_*(\pi _{\widehat A}^*L^{\vee }\otimes \mathcal
  L)\to R^g {\sigma}_*(\pi _{\widehat A}^*L^{\vee }\otimes \mathcal L),
  $$
  and similarly by the Grothendieck spectral sequence associated to $\lambda_*\circ
  {\pi_X}_*$ there exists a natural morphism
  $$
  R^g {\sigma}_*(\pi _{\widehat A}^*L^{\vee }\otimes \mathcal L)\to
  \lambda_*R^g{\pi_X}_*(\pi _{\widehat A}^*L^{\vee }\otimes \mathcal L).
  $$
  Combining the above three morphisms gives a natural morphism
  $$
  \widehat{L^{\vee }}\to \lambda_*R^g {\pi _X}_*(\pi _{\widehat A}^*L^{\vee }\otimes
  \mathcal L) = \lambda_*R^g\Psi (L^{\vee }),
  $$
  and hence by adjointness a natural morphism,
  $$
  \eta: \lambda^*\widehat {L^{\vee }}\to R^g\Psi (L^{\vee }).
  $$ 
  For any point $x\in X$, by cohomology and base change, the induced morphism on the
 {fiber} over $x$ is an isomorphism:
  \begin{multline*}
    \eta _x: 
    \lambda^*\widehat {L^{\vee }} 
    \otimes \kappa(x) \simeq H^g(\lambda(x)\times \widehat A,
    L^{\vee}\otimes    \mathcal P _{\lambda(x)})\overset{\simeq}{\longrightarrow} \\
    \overset{\simeq}{\longrightarrow} H^g(x\times \widehat A, L^{\vee }\otimes
    \mathcal L_{x})\simeq 
    R^g\Psi (L^{\vee })
    \otimes \kappa(x).
  \end{multline*}
   Therefore
  $\eta_x$ is an isomorphism for all $x\in X$ and hence $\eta$ is an isomorphism.
\end{proof}

\section{Examples}

\begin{notation}
  Let $T\subseteq \P^n$ be a projective variety. The cone over $T$ in $\A^{n+1}$ will
  be denoted by $C(T)$. In other words, if $T\simeq \Proj S$, then $C(T)\simeq \Spec
  S$.

  Linear equivalence between (Weil) divisors is denoted by $\sim$ and strict
  transform of a subvariety
  $T$ by the inverse of a birational morphism $\sigma$ is denoted by
  $\sigma^{-1}_*T$.
\end{notation}

\begin{ex}
  \label{ex:meta}
  Let $k$ be an algebraically closed field, $V\subseteq \P^n$ and $W\subseteq \P^m$
  two smooth projective varieties over $k$, and $p\in V$ a closed point.  Let
  $x_0,\dots,x_n$ and $y_0,\dots,y_m$ be homogenous coordinates on $\P^n$ and $\P^m$
  respectively.

  Consider the embedding $V\times W\subset \P^N$ induced by the Segre embedding of
  $\P^n\times \P^m$. We may choose homogenous coordinates $z_{ij}$ for $i=0,\dots,n$
  and $j=0,\dots,m$ on $\P^N$ and in these coordinates $\P^n\times \P^m$ is defined
  by the equations $z_{\alpha\gamma}z_{\beta\delta}-z_{\alpha\delta}z_{\beta\gamma}$
  for all $0\leq \alpha,\beta\leq n$ and $0\leq \gamma,\delta\leq m$.

  Next let $H\subset W$ such that $\{p\}\times H \subset \{p\}\times W$ is a
  hyperplane section of $\{p\}\times W$ in $\P^N$. Let $Y=C(V\times W)\subset
  \A^{N+1}$ and $Z=C(V\times H)\subset Y$ and let $v\in Z\subset Y$ denote the common
  vertex of $Y$ and $Z$.  If $\dim W=0$, then $H=\emptyset$. In this case let
  $Z=\{v\}$ the vertex of $Y$.  Finally let $\mathfrak m_v$ denote the ideal of $v$
  in the affine coordinate ring of $Y$. It is generated by all the variables
  $z_{ij}$.
\end{ex}

\begin{prop}\label{claim:Ex-of-blow-up}
  Let $f:X\to Y$ be the blowing up of $Y$ along $Z$.  Then $f$ is an isomorphism over
  $Y\setminus\{v\}$ and the scheme theoretic preimage of $v$ (whose support is the
  exceptional locus) is isomorphic to $V$:
  $$f^{-1}(v)\simeq V.$$
\end{prop}

\begin{proof}
  As $Z$ is of codimension $1$ in $Y$ and $Y\setminus\{v\}$ is smooth, it follows
  that $Z\setminus\{v\}$ is a Cartier divisor in $Y\setminus\{v\}$ and hence $f$ is
  indeed an isomorphism over $Y\setminus\{v\}$.

  To prove the statement about the exceptional locus of $f$, first assume that
  $V=\P^n$, $W=\P^m$, $p=[1:0:\dots:0]$, and $\{p\}\times H=(z_{0m}=0)\cap
  (\{p\}\times W)$. Then $H=(y_m=0)\subseteq W$ and hence $I=I(Z)$, the ideal of $Z$
  in the affine coordinate ring of $Y$, is generated by $\{z_{im}\vert
  i=0,\dots,n\}$. Then by the definition of blowing up, $X=\Proj \oplus_{d\geq 0}
  I^d$ and $f^{-1}v\simeq \Proj \oplus_{d\geq 0} I^d/I^{d}\mathfrak m_v$.

  Notice that $I^d/I^{d}\mathfrak m_v$ is a $k$-vector space generated by the degree
  $d$ monomials in the variables $\{z_{im}\vert i=0,\dots,n\}$. It follows that the
  graded ring $\oplus_{d\geq 0} I^d/I^{d}\mathfrak m_v$ is nothing else but
  $k[z_{im}\vert i=0,\dots,n]$ and hence $f^{-1}v\simeq \P^n=V$, so the claim is
  proved in this case.

  Next consider the case when $V\subseteq \P^n$ is arbitrary, but $W=\P^m$. In this
  case the calculation is similar, except that we have to account for the defining
  equations of $V$. They show up in the definition of the coordinate ring of $Y$ in
  the following way: If a homogenous polynomial $g\in k[x_0,\dots,x_n]$ vanishes on
  $V$ (i.e., $g\in I(V)_h$), then define $g_\gamma\in k[z_{ij}]$ for any
  $0\leq\gamma\leq m$ by replacing $x_\alpha$ with $z_{\alpha\gamma}$ for each
  $0\leq\alpha\leq n$. Then $\{g_\gamma \vert 0\leq\gamma\leq m, g\in I(V)_h\}$
  generates the ideal of $Y$ in the affine coordinate ring of $C(\P^n\times\P^m)$.
  It follows that the above computation goes through the same way, except that the
  variables $\{z_{im}\vert i=0,\dots,n\}$ on the exceptional $\mathbb P^n$ are
  subject to the equations $\{g_m \vert g\in I(V)_h\}$. However, this simply means
  that the exceptional locus of $f$, i.e., $f^{-1}v$, is cut out from $\P^n$ by these
  equations and hence it is isomorphic to $V$.

  Finally, consider the general case. The way $W$ changes the setup is the same as
  what we described for $V$. If a homogenous polynomial $h\in k[y_0,\dots,y_m]$
  vanishes on $W$ (i.e., $h\in I(W)_h$), then define $h_\alpha\in k[z_{ij}]$ for any
  $0\leq\alpha\leq n$ by replacing $y_\gamma$ with $z_{\alpha\gamma}$ for each
  $0\leq\gamma\leq m$. Then $\{h_\alpha \vert 0\leq\alpha\leq n, h\in I(W)_h\}$
  generates the ideal of $Y$ in the affine coordinate ring of $C(V\times\P^m)$.

  However, in this case, differently from the case of $V$, we do not get any
  additional equations. Indeed, we chose the coordinates so that $H=(y_m=0)$ and
  hence $y_m\not\in I(W)$, which means that we may choose the rest of the coordinates
  such that $[0:\dots:0:1]\in W$. This implies that no polynomial in the ideal of $W$
  may have a monomial term that is a constant multiple of a power of $y_{m}$. It
  follows that, since $I=I(Z)$ is generated by the elements $\{z_{im}\vert
  i=0,\dots,n\}$, any monomial term of any polynomial in the ideal of $Y$ in the
  affine coordinate ring of $C(V\times\P^m)$ that lies in $I^d$ for some $d>0$, also
  lies in $I^d\mathfrak m_v$.  Therefore these new equations do not change the ring
  $\oplus I^d/I^d\mathfrak m_v$ and so $f^{-1}v$ is still isomorphic to $V$.
\end{proof}
  
\begin{notation}
  We will use the notation introduced in \eqref{claim:Ex-of-blow-up} for $X$, $Y$,
  $Z$, and $f$. We will also use $X_{\P}$, $Y_{\P}$, $Z_{\P}$, and $f_{\P}:X_{\P}\to
  Y_{\P}$ to denote the same objects in the case $W=\P^m$, i.e., $Y_{\P}=C(V\times
  \P^m)$, $Z_{\P}=C(V\times H)$ where $H\subset \P^m$ is such that $\{p\}\times H
  \subset \{p\}\times \P^m$ is a hyperplane section of $\{p\}\times \P^m$ in $\P^N$.
\end{notation}

\begin{cor}
  $f_{\P}$ is an isomorphism over $Y_{\P}\setminus\{v\}$ and the scheme theoretic
  preimage of $v$ (whose support is the exceptional locus) via $f_{\P}$ is isomorphic
  to $V$:
  $$f_{\P}^{-1}v\simeq V.$$
\end{cor}

\begin{proof}
  This was proven as an intermediate step in, and is also straightforward from
  \eqref{claim:Ex-of-blow-up} by taking $W=\P^m$.
\end{proof}

\begin{prop}
  \label{claim:CM}
  Assume that $V$ and $W$ are both positive dimensional, $W\subseteq \P^m$ is a
  complete intersection, and the embedding $V\times \P^r\subset \P^N$ for any linear
  subvariety $\P^r\subseteq \P^m$ induced by the Segre embedding of $\P^n\times \P^m$
  is projectively normal.  Then $X$ is Gorenstein.
\end{prop}

\begin{proof}
  First note that the projective normality assumption implies that
  $Y_{\P}=C(V\times\P^m)$ is normal and hence we may consider divisors and their
  linear equivalence on it.

  Let $H'\subset \P^m$ be an arbitrary hypersurface (different from $H$ and not
  necessarily linear). Observe that $H'\sim d\cdot H$ with $d=\deg H'$, so $V\times
  H'\sim d\cdot (V\times H)$, and hence $C(V\times H')\sim d\cdot C(V\times H)$ as
  divisors on $Y_{\P}$.

  Since $f_{\P}$ is a small morphism it follows that the strict transforms of these
  divisors on $X_{\P}$ are also linearly equivalent: $f^{-1}_*C(V\times H')\sim
  d\cdot f^{-1}_*C(V\times H)$ (where by abuse of notation we let $f=f_{\P}$).  By the basic properties of blowing up, the (scheme-theoretic) pre-image of $C(V\times H)$ is a Cartier divisor on $X$ which coincides with $f^{-1}_*C(V\times H)$ (as $f$ is small).  However, then $f^{-1}_*C(V\times H')$
  is also a Cartier divisor and hence it is Gorenstein if and only if $X_{\P}$ is.
  Note that $f^{-1}_*C(V\times H')$ is nothing else but the blow up of $C(V\times
  H')$ along $C(V\times (H'\cap H))$.

  By assumption $W$ is a complete intersection, so applying the above argument for
  the intersection of the hypersurfaces cutting out $W$ shows that $X$ is Gorenstein
  if and only if $X_{\P}$ is Gorenstein. In other words, it is enough to prove the
  statement with the additional assumption that $W=\P^m$. In particular, we have
  $X=X_{\P}$, etc.
  
  In this case the same argument as above shows that the statement holds for $m$ if
  and only if it holds for $m-1$, so we only need to prove it for $m=1$.  In that
  case $H\in \P^1$ is a single point. Choose another point $H'\in \P^1$. As above,
  $f^{-1}_*C(V\times H')$ is a Cartier divisor in $X$ and it is the blow up of
  $C(V\times H')$ along the intersection $C(V\times H')\cap C(V\times H)$.

  We claim that this intersection is just the vertex of $C(V)$.
   
  To see this, view $Y=Y_{\P}=C(V\times \P^1)$ as a subscheme of $C(\P^n\times
  \P^1)$.  Inside $C(\P^n\times \P^1)$ the cones $C(\P^n\times H)$ and $C(\P^n\times
  H')$ are just linear subspaces of dimension $n+1$ whose scheme theoretic
  intersection is the single reduced point $v$.  Therefore we have that
  $$
  C(V\times H')\cap C(V\times H)\subseteq C(\P^m\times H')\cap C(\P^m\times H)=\{v\}
  $$
  proving the same for this intersection.

  Finally then $f^{-1}_*C(V\times H')$, the blow up of $C(V\times H')$ along the
  intersection $C(V\times H')\cap C(V\times H)$ is just the blow up of $C(V)$ at its
  vertex and hence it is smooth and in particular Gorenstein. This completes the
  proof.
\end{proof}

\begin{lem}\label{lem:proj-normal}
  Let $V\subseteq \P^n$ and $W\subseteq \P^m$ be two normal complete intersection
  varieties of positive dimension.  Assume that either $\dim V+\dim W>2$ or if $\dim
  V=\dim W=1$, then $n=m=2$. Then the embedding $V\times W\subset \P^N$ induced by
  the Segre embedding of $\P^n\times \P^m$ is projectively normal.
\end{lem}

\begin{proof}
  It follows easily from the definition of the Segre embedding, that it is itself
  projectively normal and hence it is enough to prove that
  \begin{equation}
    \label{eq:1}
    H^0(\P^n\times \P^m, \sO_{\P^N}(d)|_{\P^n\times \P^m})\to   H^0(V\times W,
    {\sO_{\P^N}(d)}|_{V\times W} )
  \end{equation}
  is surjective for all $d\in\N$.

  We prove this by induction on the combined number of hypersurfaces cutting out $V$
  and $W$. When this number is $0$, then $V=\P^n$ and $W=\P^m$ so we are done.

  Otherwise, assume that $\dim V\leq \dim W$ and if $\dim V=\dim W=1$ then $\deg
  V=e\geq \deg W$.  Let $V'\subseteq \P^n$ be a complete intersection variety of
  dimension $\dim V+1$ such that $V= V'\cap H'$ where $H'\subset \P^n$ is a
  hypersurface of degree $e$. Then $V\times W\subset V'\times W$ is a Cartier divisor
  with ideal sheaf $\sI\simeq\pi_1^*\sO_{V'}(-e)$ where $\pi_1:V'\times W\to V'$ is the
  projection to the first factor.  It follows that for every $d\in\N$ there exists a
  short exact sequence,
  \begin{equation*}
    0\to {\sO_{\P^N}(d)}|_{V'\times W}\otimes \pi_1^*\sO_{V'}(-e) \to
    {\sO_{\P^N}(d)}|_{V'\times W} \to     {\sO_{\P^N}(d)}|_{V\times W} \to 0,
  \end{equation*}
  and hence an induced exact sequence of cohomology
  \begin{multline*}
    H^0(V'\times W, {\sO_{\P^N}(d)}|_{V'\times W} ) \to
    H^0(V\times W, {\sO_{\P^N}(d)}|_{V\times W} ) \to \\ \to
    H^1(V'\times W, \pi_1^*\sO_{V'}(d-e)\otimes \pi_2^*\sO_{W}(d) ),
  \end{multline*}
  where $\pi_2:V'\times W\to W$ is the projection to the second factor. 

  Since by assumption $V'$ is a complete intersection variety of dimension at least
  $2$, it follows that $H^1(V', \sO_{V'}(d-e))=0$. 

  If $\dim W>1$, then it follows similarly that $H^1(W,\sO_{W}(d))=0$.

  If $\dim W=1$, then since $0<\dim V\leq \dim W$ we also have $\dim V=1$. By
  assumption $V$ and $W$ are normal and hence regular, and in this case we assumed
  earlier that $\deg V=e\geq \deg W$. It follows that as long as $e>d$, then $H^0(V',
  \sO_{V'}(d-e))=0$ and if $e\leq d$, then $d\geq \deg W$ and hence
  $H^1(W,\sO_{W}(d))=0$.
  
  In both cases we obtain that by the K\"unneth formula (cf.\ \cite[(6.7.8)]{EGAIII},
  \cite[9.2.4]{MR1252397}),
  $$
  H^1(V'\times W, \pi_1^*\sO_{V'}(d-e)\otimes \pi_2^*\sO_{W}(d) )=0,
  $$
  and hence 
  $$
  H^0(V'\times W, {\sO_{\P^N}(d)}|_{V'\times W} ) \to H^0(V\times W,
  {\sO_{\P^N}(d)}|_{V\times W} )
  $$
  is surjective. By induction we may assume that 
  $$
  H^0(\P^n\times \P^m, \sO_{\P^N}(d)|_{\P^n\times \P^m})\to H^0(V'\times W,
  {\sO_{\P^N}(d)}|_{V'\times W} )
  $$
  is surjective, so it follows that the desired map in (\ref{eq:1}) is surjective as
  well and the statement is proven.
\end{proof}

\begin{cor}
  \label{cor:Gorenstein}
  Assume that $V\subseteq \P^n$ and $W\subseteq \P^m$ are two positive dimensional
  normal complete intersection varieties.  Then $X$ is Gorenstein.
\end{cor}

\begin{proof}
  Follows by combining \eqref{claim:CM} and \eqref{lem:proj-normal}. Note that in
  \eqref{claim:CM} the embedding $V\times W\hookrightarrow \P^N$ does not need to be
  projectively normal, only $V\times \P^r\hookrightarrow \P^N$ does, which indeed
  follows from \eqref{lem:proj-normal}.
\end{proof}

\begin{ex}\label{e-1} 
  Let $k$ be an algebraically closed field.  We will construct a birational
  projective morphism $f:X\to Y$ such that $X$ is Gorenstein (and log canonical) and
  $R^1f_*\omega _X \ne 0$.

  Let $E_1,E_2\subseteq \P^2$ be two smooth projective cubic curves. Consider the
  construction in \eqref{ex:meta} with $V=E_1$, $W=E_2$.  As in that construction let
  $f:X\to Y$ be the blow up of $Y=C(E_1\times E_2)$ along $Z=C(E_1\times H)$ where
  $H\subseteq E_2$ is a hyperplane section.  The common vertex of $Y$ and $Z$ will
  still be denoted by $v\in Z\subset Y$.  The map $f$ is an isomorphism over
  $Y\setminus\{v\}$ and $f^{-1}v\simeq E_1$ by \eqref{claim:Ex-of-blow-up}.

  \begin{prop}
    Both $X$ and $Y$ are smooth in codimension $1$ with trivial canonical divisor and
    $X$ is Gorenstein and hence Cohen-Macaulay.
  \end{prop}

  \begin{proof}
    By construction $Y\setminus\{v\}\simeq X\setminus f^{-1}v$ is smooth, so the
    first statement follows. Furthermore, $Y\setminus\{v\}\simeq X\setminus f^{-1}v$
    is an affine bundle over $E_1\times E_2$, so by the choice of $E_1$ and $E_2$,
    the canonical divisor of $Y\setminus\{v\}\simeq X\setminus f^{-1}v$ is trivial.
    However, the complement of this set has codimension at least $2$ in both $X$ and
    $Y$ and hence their canonical divisors are trivial as well.
    Since $E_1,E_2\subset \P^2$ are hypersurfaces, $X$ is Gorenstein by
    \eqref{cor:Gorenstein}.
  \end{proof}

  Let $E$ denote $f^{-1}v$. So we have that $E\simeq E_1$ and there is a short exact
  sequence
  $$
  0\to \mathscr I_E\to \mathcal O _X \to \mathcal O _{E}\to 0.
  $$
  Pushing this forward via $f$ we obtain a homomorphism $\phi:R^1f_*\mathcal O _X \to
  R^1f_*\mathcal O _{E}$. Since the maximum dimension of any fiber of $f$ is $1$, we
  have $R^2f_*\mathcal I_E=0$.  It follows that $R^1f_* \omega _X=R^1f_*\mathcal
  O_X\ne 0$, because $R^1f_*\mathcal O_{E}\ne 0$ (it is a sheaf supported on $v$ of
  length $h^1(\mathcal O _{E})=1$).
\end{ex}

\begin{ex}\label{e-2} 
  Let $k$ be an algebraically closed field of characteristic $p\ne 0$. Then there
  exists a birational morphism $f:X\to Y$ of varieties (defined over $k$) such that
  $X$ is smooth of dimension $7$ and $R^if_*\omega _X \ne 0$.  for some $i\in \{
  1,2,3,4,5 \}$.

  Let $Z$ be a smooth $6$-dimensional variety and $L$ a very ample line bundle such
  that $H^1(Z,\omega _Z\otimes L)\ne 0$.
  (such varieties exist by \cite{LR97}).  By Serre vanishing $H^i(Z,\omega _Z\otimes
  L^j)=0$ for all $i>0$ and $j\gg 0$. Let $m$ be the largest positive integer such
  that $H^i(Z,\omega _Z\otimes L^m)\ne 0$ for some $i>0$.

  After replacing $L$ by $L^m$ we may assume that there exists a $q>0$ such that
  $H^q(Z,\omega _Z\otimes L)\neq 0$, but $H^i(Z,\omega _Z\otimes L^j)= 0$ for all
  $i>0$ and $j\geq 2$.  Note that $q<6$, because $H^6(Z,\omega _Z\otimes L)$ is dual
  to $H^0(Z,L^{-1})=0$.

  Let $Y$ be the cone over the embedding of $Z$ given by $L$, $f:X\to Y$ the blow up
  of the vertex $v\in Y$, and $E=f^{-1}v$ the exceptional divisor of $f$.  Note that
  $E\simeq Z$ and $\omega _E(-jE)\simeq \omega _Z\otimes L^j$ for any $j$.

For $j\geq 1$ consider the short exact sequence
  $$
  0\to \omega _X(-jE)\to \omega_X(-(j-1)E)\to \omega _E(-jE)\to 0.
  $$
  \begin{subclaim}
    $R^if_*\omega _X(-E)=0$ for all $i>0$ and $R^if_*\omega _X=0$ for all $i>0$,
    such that $H^i(Z,\omega _Z\otimes L)= 0$.
  \end{subclaim}

  \begin{proof}[Proof of Claim]
    As $-E$ is $f$-ample we have, by Serre vanishing again, that $R^if_*\omega
    _X(-jE)=0$ for all $i>0$ and some $j>0$. If either $j>1$ or $j=1$ and
    $H^i(Z,\omega _Z\otimes L)= 0$, then $R^if_*\omega _E(-jE)=H^i(Z,\omega_Z\otimes
    L^j)=0$ by the choice of $L$.  Therefore, the exact sequence
    $$
    0= R^if_*\omega_X(-jE)\to R^if_*\omega _X(-(j-1)E)\to R^if_*\omega _E(-jE)=0
    $$
    gives that $R^if_*\omega _X(-(j-1)E)=0$. The claim follows  by induction.
  \end{proof}

  From the above claim it follows that 
  $$
  0= R^qf_*\omega _X(-E) \to R^qf_*\omega_X\to R^qf_*\omega _E(-E)\to
  R^{q+1}f_*\omega _X(-E)=0
  $$
  Since $R^qf_*\omega _E(-E)=H^q(Z,\omega_Z\otimes L)\neq 0$, we obtain that
  $R^qf_*\omega_X\neq 0$ as claimed.
  \qed
\end{ex}
\begin{rem}\label{ex:3.5}
  The above example is certainly well known (see for example \cite[4.7.2]{CR11b}) and
  one can easily construct examples in dimension $\geq 3$ (using for example the
  results of \cite{Raynaud78} and \cite{Mukai79}).  We have chosen to include the
  above example because of its elementary nature.
\end{rem}

\begin{prop}\label{p-1} 
  There exists a variety $T$ and a generically finite projective separable morphism
  to an abelian variety $\lambda :T\to A$ defined over an algebraically closed field
  $k$ such that:
  \begin{itemize}
  \item If $\charact k=0$, then $T$ is Gorenstein (and hence 
    Cohen-Macaulay) with a single isolated log canonical singularity, and
    $R^1\lambda_*\omega _T\ne 0$, and
  \item If $\charact k=p>0$, then $T$ is smooth, and $R^i\lambda_*\omega _T\ne 0$ for
    some $i>0$.
  \end{itemize}
\end{prop}
\begin{proof} 
  First assume that $\charact k=0$ and let $f:X\to Y$ be as in \eqref{e-1}. We may
  assume that $X$ and $Y$ are projective.  Let $X'\to X$ and $Y'\to Y$ be birational
  morphisms that are isomorphisms near $f^{-1}(v)$ and $v$ respectively such that
  there is a birational morphism $f':X'\to Y'$ and a generically finite morphism
  $g:Y'\to \mathbb P ^n$. We let $v'\in Y'$ be the inverse image of $v\in Y$ and
  $p\in \mathbb P ^n$ its image.  We may assume that there is an open subset $\mathbb
  P ^n_0\subset \mathbb P ^n$ such that $g|_{Y'_0}$ is finite where
  $Y'_0=g^{-1}(\mathbb P ^n_0)$.  Note that if we let $X'_0$ be the inverse image of
  $Y'_0$ and $g'=g\circ f'$, then we have $R^ig'_*\omega
  _{X'_0}=g_*R^if'_*\omega_{X'_0}$.

  Let $A$ be an $n$-dimensional abelian variety, $A'\to A$ a birational morphism of
  smooth varieties and $A'\to \mathbb P ^n$ a generically finite morphism. We may
  assume that there are points $a'\in A'$ and $a\in A$ such that $(A',a')\to (A,a)$
  is locally an isomorphism and $(A',a')\to (\mathbb P ^n,p)$ is locally \'etale.

  Let $U$ be the normalization of the main component of $X'\times _{\mathbb P ^n}A'$
  and $h:U\to X'$ the corresponding morphism.  We let $E\subset (f'\circ
  h)^{-1}(v')\subset U$ be the component corresponding to $(v',a')\in Y'\times
  _{\mathbb P ^n}A'$. Then, the morphism $(U,E)\to (Y'\times _{\mathbb P
    ^n}A',(v',a'))\to (A,a)$ is \'etale locally (on the base) isomorphic to
  $(X,f^{-1}(v))\to (Y,v)\to (\mathbb P ^n,p)$.

  Let $\nu: T\to U$ be a birational morphism such that $\nu$ is an isomorphism over a
  neigborhood of $E\subset U$ and $T\setminus \nu ^{-1}(E)$ is smooth. Let $\lambda
  :T\to A$ be the induced morphism. It is clear from what we have observed above that
  $\lambda(E)$ is one of the components of the support of $R^1\lambda_*\omega _T\ne
  0$ and $T$ has the required singularities.



  Assume now that the $\charact k=p>0$ and let $f:X\to Y$ be a birational morphism of
  varieties such that $X$ is smooth and $R^if_*\omega _X \ne 0$ for some $i>0$. This
  $i$ will be fixed for the rest of the proof.  The existence of such morphisms is
  well-known (see \eqref{ex:3.5}) and an explicit example in dimension $7$ is given
  in \eqref{e-2}. Further let $A$ be an abelian variety of the same dimension as $X$
  and $Y$ and set $n=\dim A=\dim X=\dim Y$.  There are embeddings $Y\subset \mathbb P
  ^{m_1}$, $A\subset \mathbb P ^{m_2}$ and $\mathbb P ^{m_1}\times \mathbb P
  ^{m_2}\subset \mathbb P ^{M}$.  Let $H$ be a very ample divisor on $\mathbb P ^{M}$
  and $U\subset Y\times A$ the intersection of $n$ general members $H_1,\ldots
  ,H_n\in |H|$ with $Y\times A$. By choice the induced maps $h:U\to Y$ and $a:U\to A$
  are generically finite, $U$ intersects $v\times A$ transversely so that $V=U\cap
  (v\times A)$ is a finite set of reduced points and $U\setminus V$ is smooth by
  Bertini's theorem (cf. \cite[II.8.18]{Hartshorne77} and its proof).  It follows
  that any singular point $u\in U$ is a point in $V$ and $(U,u)$ is locally
  isomorphic to $(Y,v)$.  We claim that $a$ is finite in a neighborhood of $u\in U$.
  Consider any contracted curve i.e. any curve $C\subset U\cap (Y\times a(u))$. We
  must show that $u\not\in C$.  Let $\nu :T\to U$ be the blow up of $U$ along $V$ and
  $\tilde C$ the strict transform of $C$ on $T$.  We let $\mu:{\rm Bl}_V\mathbb P
  ^M\to \mathbb P ^M$, $E=\mu ^{-1}(u)\cong \mathbb P ^{M-1}$ and we denote $h_i=\mu
  ^{-1}_*H_i |_E$ the corresponding hyperplanes.  To verify the claim it suffices to
  check that $\nu ^{-1}(u)\cap \tilde C=\emptyset$.  But this is now clear as $\nu
  ^{-1}(u)\cong Z\subset \mathbb P ^{M-1}$ and the $h_i$ are general hyperplanes so
  that $Z\cap h_1\cap \ldots \cap h_n=\emptyset$ as $Z$ is $(n-1)$-dimensional.

  Let $\lambda =a\circ \nu :T\to A$ be the induced morphism.  By construction the
  support of the sheaf $R^i\nu _* \omega _T$ is $V$.
  Since $a$ is finite on a neighborhood of $u\in U$, it follows that $0\ne a_*R^i\nu
  _* \omega _T\subset R^i\lambda _* \omega _T$ and hence $R^i\lambda _* \omega _T\neq
  0$ for the same $i>0$.
 \end{proof}

\section{Main result}

\begin{prop}\label{p-2} Assume that $\lambda :X\to A$ is generically finite on to its image where $X$ is  a projective Cohen Macaulay variety and $A$ is an abelian variety. If ${\rm char }(k)=p>0$, then we assume that there is an ample line bundle $L$ on $A$ whose degree is not divisible by $p$.
  If $R^i\pi _{\widehat A*}\mathcal L=0$ for all $i<n$, then $R^i\lambda_*\omega _X=0$
  for all $i>0$.
\end{prop}

\begin{proof}
  By Theorem A of \cite{PP11}, $R^i\Phi (\mathcal O _X)= R^i{\pi_{\widehat
      A}}_*\mathcal L=0$ for all $i<n$, is equivalent to
  $$
  H^i(X,\omega _X\otimes R^g\Psi (L^{\vee}))=0\qquad \forall\ i>0,
  $$
  where $L$ is sufficiently ample on $\widehat A$ and $R^g\Psi (L^{\vee })=\lambda^*\widehat
  {L^{\vee} }$ (cf. \eqref{l-c}). It is easy to see that this is in turn equivalent to
  $$
  H^i(X,\omega _X\otimes \lambda^* (\widehat{t_{\widehat a}^*L^{\vee}}))=0\qquad \forall\
  i>0,\ \forall\ \widehat a\in \widehat A,
  $$
  where $L$ is sufficiently ample on $\widehat A$. By \cite[3.1]{Mukai81}, we have
  $\widehat{t_{\widehat a}^*L^{\vee}}=\widehat{L^{\vee}}\otimes P_{-\widehat a}$ and
  hence $H^i(X,\omega _X\otimes \lambda^* (\widehat{L^{\vee}}\otimes P_{-\widehat a}))=0$.
  Thus, by cohomology and base change, we have that
  $$
  \mathbf R \widehat S (\mathbf R \lambda_* \omega _X\otimes \widehat {L^{\vee}
  })=^{\eqref{eq-1}}\mathbf R \Phi (\omega _X\otimes \lambda^* \widehat {L^{\vee} })= R^0
  \Phi (\omega _X\otimes \lambda^* \widehat {L^{\vee} }).
  $$ 
  In particular $\mathbf R \lambda_* \omega _X\otimes \widehat {L^{\vee} }$ is WIT-$0$.
  \begin{claim}
    For any ample line bundle $M$ on $A$, we have that 
    $$
    H^i(X,\omega _X\otimes \lambda^*(\widehat {L^{\vee} }\otimes M\otimes P_{-\widehat{a}}
    ))=0\qquad \forall\ i>0,\ \forall\ \widehat a\in \widehat A.
    $$
  \end{claim}

  \begin{proof}
    We follow the argument in \cite[2.9]{PP03}.  For any $P=P_{-\widehat a} $, we
    have
    \begin{multline*}
      H^i(X, \omega _X\otimes \lambda^*(\widehat {L^{\vee} }\otimes M\otimes P))=
      R^i\Gamma (X, \omega _X\otimes \lambda^*(\widehat {L^{\vee} }\otimes M\otimes
      P))=^{\rm P.F.}
      \\
      R^i\Gamma (A, \mathbf R \lambda_*\omega _X\otimes \widehat {L^{\vee} }\otimes
      M\otimes P)= {\rm Ext}_{D(A)}^i((M\otimes P)^\vee, \mathbf R \lambda_*\omega
      _X\otimes \widehat {L^{\vee} })=^{\eqref{eq-0}}
      \\
      {\rm Ext}_{D(\widehat A)}^{i+g}(R^g\widehat S({(M\otimes P)^\vee}), R^0\Phi
      (\omega _X\otimes \lambda^*\widehat {L^{\vee} }))=
      \\
      H^{i+g}(\widehat A, R^0\Phi(\omega _X\otimes \lambda^*\widehat {L^{\vee} })\otimes
      R^g\widehat S({(M\otimes P)^\vee})^\vee )=0\qquad i>0.
    \end{multline*}
    (The third equality follows as $M\otimes P$ is free, the fifth follows since
    $R^g\widehat S({M\otimes P)^\vee}$ is free and the last one since $i+g>g=\dim \widehat A$.)
\end{proof}

Let $\phi _L:\widehat A\to A$ be the isogeny induced by $\phi _L(\widehat
x)=t_{\widehat x}^*L\otimes L^\vee$, then $\phi _L ^* \widehat {L^{\vee}}=L^ {\oplus h^0(L)}$. 
We may assume that the characteristic does not divide the degree of $L$ so that $\phi _L$ is separable.
Let $X'=X\times _A \widehat A$, $\phi :X'\to X$ and $\lambda ':X'\to \widehat
A$ the induced morphisms. Note that ${\phi }_*\mathcal O_{X'}=\lambda^*({\phi _L}_*\mathcal
O_{\widehat A})=\lambda^*(\oplus P_{\alpha _i})$ where the $\alpha _i$ are the elements in
$K\subset \widehat A$, the kernel of the induced homomorphism $\phi _L:\widehat A
\to A$.  By the above equation and flat base change
$$
H^i(X',\omega _{X'}\otimes {\lambda'}^*\phi _L^*(\widehat {L^{\vee} }\otimes M))=\bigoplus
_{\alpha \in K} H^i(X,\omega _X \otimes\lambda^*(\widehat {L^{\vee} }\otimes M\otimes
P_{\alpha }))=0
$$ 
for all $i>0.$ But then $H^i(X',\omega _{X'}\otimes {\lambda'}^*(L\otimes \phi _L^* M))=0$
for all $i>0$.  Note that if $M$ is sufficiently ample on $A$ then so is $L\otimes
\phi _L^* M$ on $\widehat A$.  It follows by an easy (and standard) spectral sequence
argument that $R^i\lambda'_* \omega _{X'}=0$ for $i>0$. Since $\omega _X$ is a summand of
$\phi _* \omega _{X'}=\mathbf R\phi _*\omega _{X'}$, and $\mathbf R\lambda_* \mathbf R\phi
_*\omega _{X'}=\mathbf R{\phi _L}_* \mathbf R\lambda'_*\omega _{X'}$, it follows that
$R^i\lambda_*\omega _X$ is a summand of $R^i\lambda_*\phi _* \omega _{X'}={\phi _L
}_*R^i\lambda'_*\omega _{X'}$ and hence $R^i\lambda_*\omega _X=0$ for all $i>0$.
\end{proof}
\begin{proof}[Proof of \eqref{t-1}] Immediate from \eqref{p-1} and \eqref{p-2}.
\end{proof}

\enddocument
\end